\title{Hypergraph Ramsey numbers}
\author{David Conlon\thanks{St John's College, Cambridge, United Kingdom.
E-mail: {\tt D.Conlon@dpmms.cam.ac.uk}. Research supported by a
Junior Research Fellowship at St John's College, Cambridge.} \and
Jacob Fox\thanks{Department of Mathematics, Princeton, Princeton,
NJ. Email: {\tt jacobfox@math.princeton.edu}. Research supported
by an NSF Graduate Research Fellowship and a Princeton Centennial
Fellowship.} \and Benny Sudakov\thanks{Department of Mathematics,
UCLA,  Los Angeles, CA 90095. Email: {\tt bsudakov@math.ucla.edu}.
Research supported in part by NSF CAREER award DMS-0546523 and by
USA-Israeli BSF grant.}}

\documentclass[11pt]{article}
\usepackage{amsfonts,amssymb,amsmath,latexsym}

\newenvironment{proof}
      {\medskip\noindent{\bf Proof.}\hspace{1mm}}
      {\hfill$\Box$\medskip}

\def\qed{\ifvmode\mbox{ }\else\unskip\fi\hskip 1em plus 10fill$\Box$}

\newtheorem{theorem}{Theorem}[section]

\newtheorem{lemma}[theorem]{Lemma}
\newtheorem{proposition}[theorem]{Proposition}
\newtheorem{corollary}[theorem]{Corollary}
\newtheorem{conjecture}[theorem]{Conjecture}

\newtheorem{question}[theorem]{Question}

\begin{document}
\date{}
\maketitle

\begin{abstract}
The Ramsey number $r_k(s,n)$ is the minimum $N$ such that every
red-blue coloring of the $k$-tuples of an $N$-element set contains
either a red set of size $s$ or a blue set of size $n$, where a
set is called red (blue) if all $k$-tuples from this set are red
(blue). In this paper we obtain new estimates for several basic
hypergraph Ramsey problems. We give a new upper bound for
$r_k(s,n)$ for $k \geq 3$ and $s$ fixed. In particular, we show
that $$r_3(s,n) \leq 2^{n^{s-2}\log n},$$ which improves by a
factor of $n^{s-2}/\textrm{polylog}\,n$ the exponent of the
previous upper bound of Erd\H{o}s and Rado from 1952. We also
obtain a new lower bound for these numbers, showing that there are
constants $c_1,c_2>0$
 such that $$r_3(s,n) \geq 2^{c_1 \, sn \, \log (n/s)}$$ for all $4 \leq s
\leq c_2n$. When $s$ is a constant, it gives the first
superexponential lower bound for $r_3(s,n)$, answering an open
question posed by Erd\H{o}s and Hajnal in 1972. Next, we consider
the $3$-color Ramsey number $r_3(n,n,n)$, which is the minimum $N$
such that every $3$-coloring of the triples of an $N$-element set
contains a monochromatic set of size $n$. Improving another old
result of Erd\H{o}s and Hajnal, we show that $$r_3(n,n,n) \geq
2^{n^{c \log n}}.$$ Finally, we make some progress on related
hypergraph Ramsey-type problems.
\end{abstract}

\section{Introduction}
Ramsey theory refers to a large body of deep results in mathematics
whose underlying philosophy is captured succinctly by the statement that ``Every large system,
contains a large well organized subsystem.'' This is an area in which a great
variety of techniques from many branches of mathematics are used and
whose results are important not only to combinatorics but also to logic, analysis, number theory, and
geometry. Since the publication of the seminal paper of Ramsey
in 1930, this subject experienced tremendous growth,
and is currently among the most active areas in combinatorics.

The {\it Ramsey number} $r(s,n)$ is the least integer $N$ such
that every red-blue coloring of the edges of the complete graph
$K_N$ on $N$ vertices contains either a red $K_s$ (i.e., a
complete subgraph all of whose edges are colored red) or a blue
$K_n$. Ramsey's theorem states that $r(s,n)$ exists for all $s$
and $n$. Determining or estimating Ramsey numbers is one of the
central problem in combinatorics, see the book {\it Ramsey theory}
\cite{GRS90} for details. A classical result of Erd\H{o}s and
Szekeres~\cite{ES35}, which is a quantitative version of Ramsey's
theorem, implies that $r(n,n) \leq 2^{2n}$ for every positive
integer $n$. Erd\H{o}s~\cite{E47} showed using probabilistic
arguments that $r(n,n)
> 2^{n/2}$ for $n
> 2$. Over the last sixty years, there have been several
improvements on these bounds (see, e.g., \cite{C08}). However,
despite efforts by various researchers, the constant factors in
the above exponents remain the same.

Off-diagonal Ramsey numbers, i.e. $r(s,n)$ with $s \not = n$, have
also been intensely studied. For example, after several successive
improvements, it is known (see \cite{AKS80}, \cite{K95},
\cite{S77}) that there are constants $c_1, \ldots, c_4$ such that
$$c_1\frac{n^2}{\log n} \leq r(3,n) \leq c_2 \frac{n^2}{\log n},$$
and for fixed $s>3$,
\begin{equation}
\label{eq2} c_3\left(\frac{n}{\log n} \right)^{(s+1)/2} \leq
r(s,n) \leq c_4\frac{n^{s-1}}{\log^{s-2} n},
\end{equation}
(For $s=4$, Bohman \cite{B08} recently improved the lower bound by
a factor of $\log^{1/2} n$.) All logarithms in this paper are base
$e$ unless otherwise stated.

Although already for graph Ramsey numbers there are significant gaps
between lower and upper bounds, our knowledge of hypergraph Ramsey
numbers is even weaker. The Ramsey number $r_k(s,n)$ is the minimum
$N$ such that every red-blue coloring of the unordered $k$-tuples of
an $N$-element set contains either a red set of size $s$ or a blue
set of size $n$, where a set is called red (blue) if all $k$-tuples
from this set are red (blue). Erd\H{o}s, Hajnal, and Rado
\cite{EHR65} showed that there are positive constants $c$ and $c'$
such that
$$2^{cn^2}<r_3(n,n)<2^{2^{c'n}}.$$ They also conjectured that
$r_3(n,n)>2^{2^{cn}}$ for some constant $c>0$ and Erd\H{o}s offered
a \$500 reward for a proof. Similarly, for $k \geq 4$, there is a
difference of one exponential between known upper and lower bounds
for $r_k(n,n)$, i.e., $$t_{k-1}(cn^2) \leq r_k(n,n) \leq t_k(c'n),$$
where the tower function $t_k(x)$ is defined by $t_1(x)=x$ and
$t_{i+1}(x)=2^{t_i(x)}$.

The study of $3$-uniform hypergraphs is particularly important for
our understanding of hypergraph Ramsey numbers. This is because of
an ingenious construction called the stepping-up lemma due to
Erd\H{o}s and Hajnal (see, e.g., Chapter 4.7 in \cite{GRS90}).
Their method allows one to construct lower bound colorings for
uniformity $k+1$ from colorings for uniformity $k$, effectively
gaining an extra exponential each time it is applied.
Unfortunately, the smallest $k$ for which it works is $k=3$.
Therefore, proving that $r_3(n,n)$ has doubly exponential growth
will allow one to close the gap between the upper and lower bounds
for $r_k(n,n)$ for all uniformities $k$. There is some evidence
that the growth rate of $r_3(n,n)$ is closer to the upper bound,
namely, that with four colors instead of two this is known to be
true. Erd\H{o}s and Hajnal (see, e.g., \cite{GRS90}) constructed a
$4$-coloring of the triples of a set of size $2^{2^{cn}}$ which
does not contain a monochromatic subset of size $n$. This result
shows that the number of colors matters a lot in this problem and
leads to the question of what happens in the intermediate case
when we use three colors. The $3$-color Ramsey number $r_3(n,n,n)$
is the minimum $N$ such that every $3$-coloring of the triples of
an $N$-element set contains a monochromatic set of size $n$.
Naturally, for $r_3(n,n,n)$, one should expect at least some
improvement on the $2^{cn^2}$ lower bound. Indeed, Erd\H{o}s and
Hajnal provided such a result (see \cite{EH89} and \cite{CG98}),
showing that $r_3(n,n,n) \geq 2^{c n^2 \log^2 n}$. Here, we
substantially improve this bound, extending the above mentioned
stepping-up lemma of these two authors to show

\begin{theorem} \label{threecolour}
There exists a constant $c$ such that
\begin{equation} \label{erdoshajnal}
r_3(n,n,n) \geq 2^{n^{c \log n}}.
\end{equation}
\end{theorem}

For off-diagonal Ramsey numbers, a classical argument of Erd\H{o}s
and Rado \cite{ER52} from 1952 demonstrates that
\begin{equation}\label{erdosrado} r_k(s,n) \leq 2^{{r_{k-1}(s-1,n-1) \choose
k-1}}.\end{equation} Together with the upper bound in (\ref{eq2})
it gives for fixed $s$ that $r_3(s,n) \leq 2^{{r_2(s-1,n-1)
\choose 2}} \leq 2^{c\frac{n^{2s-4}}{\log^{2s-6} n}}$. Our next
result improves the exponent of this upper bound by a factor of
$n^{s-2}/\textrm{polylog}\,n$.
\begin{theorem} \label{upperbound}
For fixed $s \geq 4$ and sufficiently large $n$,
\begin{equation}\label{mainlowerbound} \log r_3(s,n) \leq
\Big(\frac{(s-3)}{(s-2)!}+o(1)\Big)n^{s-2} \log n.\end{equation}
\end{theorem}
Clearly, a similar improvement for off-diagonal Ramsey numbers of
higher uniformity follows from this result together with
(\ref{erdosrado}).

Erd\H{o}s and Hajnal \cite{EH72} showed that $\log r_3(4,n) > cn$
using the following simple construction. They consider a random
tournament on $[N]= \{1, \ldots, N\}$ and color the triples from
$[N]$ red if they form a cyclic triangle and blue otherwise. Since
it is well known and easy to show that every tournament on four
vertices contains at most two cyclic triangles and a random
tournament on $N$ vertices with high probability does not contain
a transitive subtournament of size $c'\log N$, the resulting
coloring neither has a red set of size $4$ nor a blue set of size
$c'\log N$. In the same paper from 1972, they suggested that
probably $\frac{\log r_3(4,n)}{n} \to \infty$. Here we prove the
following new lower bound which implies this conjecture.

\begin{theorem} \label{lowerbound}
There are constants $c_1,c_2>0$ such that $$\log r_3(s,n) \geq c_1
\, sn \, \log (n/s)$$ for all $4 \leq s \leq
 c_2n$.
\end{theorem}
Combining this result together with the stepping-up lemma of
Erd\H{o}s and Hajnal (see \cite{GRS90}), one can also obtain
analogous improvements of lower bounds for off-diagonal Ramsey
numbers for complete $k$-uniform hypergraphs with $k \geq 4$.

In view of our unsatisfactory knowledge of the growth rate of
hypergraph Ramsey numbers, Erd\H{o}s and Hajnal \cite{EH72}
started the investigation of the following more general problem.
Fix positive integers $k$, $s$, and $t$. What is the smallest $N$
such that every red-blue coloring of the $k$-tuples of an
$N$-element set
 has either a red set of size $n$ or has a set of size $s$
which contains at least $t$ blue $k$-tuples? Note that when $t={s
\choose k}$ the answer to this question is simply $r_k(n,s)$.

Let ${X \choose k}$ denote the collection of all $k$-element
subsets of the set $X$. Define $f_k(N,s,t)$ to be the largest $n$
for which every red-blue coloring of ${[N] \choose k}$ has a red
$n$-element set or a set of size $s$ which contains at least $t$
blue $k$-tuples. Erd\H{o}s and Hajnal \cite{EH72} in 1972
conjectured that as $t$ increases from $1$ to ${s \choose k}$,
$f_k(N,s,t)$ grows first like a power of $N$, then at a
well-defined value $t=h_1^{(k)}(s)$, $f_k(N,s,t)$ grows like a
power of $\log N$, i.e., $f_k(N,s,h_1^{(k)}(s)-1)>N^{c_1}$ but
$f_k(N,s,h_1^{(k)}(s))<(\log N)^{c_2}$. Then, as $t$ increases
further, at $h_2^{(k)}(s)$ the function $f_k(N,s,t)$ grows like a
power of $\log \log N$ etc. and finally $f_k\left(N,s,t\right)$
grows like a power of $\log_{(k-2)} N$ for $h_{k-2}^{(k)}(s) \leq
t \leq {s \choose k}$. Here $\log_{(i)} N$ is the $i$-fold
iterated logarithm of $N$, which is defined by $\log_{(1)} N =
\log N$ and $\log_{(j+1)} N = \log (\log_{(j)} N)$.

This problem of Erd\H{o}s and Hajnal is still widely open. In
\cite{EH72} they started a careful investigation of $h_1^{(3)}(s)$
and made several conjectures which would determine this function.
We make progress on their conjectures, computing $h_1^{(3)}(s)$
for infinitely many values of $s$. We also approximate
$h_1^{(3)}(s)$ for all $s$.

In the next section, we prove Theorem \ref{upperbound} which gives
a new upper bound on off-diagonal hypergraph Ramsey numbers. Our
lower bound on $r_3(s,n)$ appears in Section 3. In Section
\ref{sectionthreecolor}, we study the $3$-color hypergraph Ramsey
numbers and prove the lower bound for $r_3(n,n,n)$. In Section
\ref{sectionnext}, confirming a conjecture of Erd\H{o}s and
Hajnal, we determine the function $h_1^{(3)}(s)$ for infinitely
many values of $s$. Finally, in the last section of the paper, we
make several additional remarks on related hypergraph Ramsey
problems. Throughout the paper, we systematically omit floor and
ceiling signs whenever they are not crucial for the sake of
clarity of presentation. We also do not make any serious attempt
to optimize absolute constants in our statements and proofs.

\section{An Upper Bound for $r_3(s,n)$}
\label{offdiagonalsection} In this section we prove the upper
bound (\ref{mainlowerbound}) on off-diagonal hypergraph Ramsey
numbers.

First we briefly discuss a classical approach to this problem by
Erd\H{o}s-Rado and indicate where it can be improved. To prove
$\log_2 r_3(s,n) \leq {r(s-1,n-1) \choose 2}$, given a red-blue
coloring $\chi$ of the triples from $[N]$, Erd\H{o}s and Rado
greedily construct a set of vertices
$\{v_1,\ldots,v_{r(s-1,n-1)+1}\}$ such that for any given pair $1
\leq i <j \leq r(s-1,n-1)$, all triples $\{v_i,v_j,v_k\}$ with $k>j$
are of the same color, which we denote by $\chi'(v_i,v_j)$. By
definition of the Ramsey number, there is either a red clique of
size $s-1$ or a blue clique of size $n-1$ in coloring $\chi'$, and
this clique together with $v_{r(s-1,n-1)+1}$ forms a red set of size
$s$ or a blue set of size $n$ in coloring $\chi$. The greedy
construction of the set $\{v_1,\ldots,v_{r(s-1,n-1)+1}\}$ is as
follows. First, pick an arbitrary vertex $v_1$ and set
$S_1=S\setminus \{v_1\}$. After having picked $\{v_1,\ldots,v_i\}$
we also have a subset $S_i$ such that for any pair $a,b$ with $1
\leq a <b \leq i$, all triples $\{v_a,v_b,w\}$ with $w \in S_i$ are
the same color. Let $v_{i+1}$ be an arbitrary vertex in $S_i$ and
set $S_{i,0}=S_i \setminus \{v_{i+1}\}$. Suppose we already
constructed $S_{i,j} \subset S_{i,0}$ such that, for every $h \leq
j$ and $w \in S_{i,j}$, all triples $\{v_h,v_{i+1},w\}$ have the
same color. If the number of edges $(v_{j+1},v_{i+1},w)$ with $w \in
S_{i,j}$ that are red is at least $|S_{i,0}|/2$, then we let
$$S_{i,j+1}=\{w:(v_{j+1},v_{i+1},w)~\textrm{is red and}~w \in
S_{i,j}\}$$ and set $\chi'(i+1,j+1)=\textrm{red}$, otherwise we let
$$S_{i,j+1}=\{w:(v_{j+1},v_{i+1},w)~\textrm{is blue and}~w \in
S_{i,j}\}$$  and set $\chi'(i+1,j+1)=\textrm{blue}$. Finally, we let
$S_{i+1}=S_{i,i}$. Notice that $\{v_1,\ldots,v_{i+1}\}$ and
$S_{i+1}$
 have the desired properties to continue the greedy algorithm. Also, for each
edge $(v_{i+1},v_{j+1})$ that we color by $\chi'$, the set $S_{i,j}$
is at most halved. So we lose a factor of at most two for each of
the ${r(s-1,n-1) \choose 2}$ edges colored by $\chi'$.\footnote{We
also lose one element from $S_i$ when we pick $v_{i+1}$, but this 
loss is rather insubstantial.}

There are two ways we are able to improve on the Erd\H{o}s-Rado
approach. Our first improvement comes from utilizing the fact that
we do not need to ensure that for every pair $i<j$, all edges
$\{v_i,v_j,v_k\}$ with $k>j$ are of the same color. That is, the
coloring $\chi'$ will not necessarily color every pair. Furthermore,
the number of edges we color by $\chi'$ will be much smaller than
the best known estimate for ${r(s-1,n-1) \choose 2}$, and this is
how we will be able to get a smaller upper bound on $r_3(s,n)$. This
idea is nicely captured using the vertex on-line Ramsey number which
we next define. Consider the following game, played by two players,
builder and painter: at step $i+1$ a new vertex $v_{i+1}$ is
revealed; then, for every existing vertex $v_j$, $j = 1, \cdots, i$,
builder decides, in order, whether to draw the edge $v_j v_{i+1}$;
if he does expose such an edge, painter has to color it either red
or blue immediately. The {\it vertex on-line Ramsey number}
$\tilde{r}(k,l)$ is then defined as the minimum number of edges that
builder has to draw in order to force painter to create either a red
$K_k$ or a blue $K_l$. In Lemma \ref{gamelemma}, we provide an upper
bound on $\tilde{r}(s-1,n-1)$ which is much smaller than the best
known estimate on ${r(s-1,n-1) \choose 2}$. Since we are losing a
factor of at most two for every exposed edge, this immediately
improves on the Erd\H{o}s-Rado bound for $r_3(s,n)$.

A further improvement can be made by using the observation that
there will not be many pairs $i<j$ for which all triples
$\{v_i,v_j,v_k\}$ with $k>j$ are red. That is, we will be able to
show that there are not many red edges in the coloring $\chi'$ we
construct. Let $0<\alpha \ll 1/2$. Suppose we have
$\{v_1,\ldots,v_i\}$ and a set $S$ and, for a given $j<i$, we want
to find a subset $S' \subset S$ such that all triples
$\{v_j,v_i,w\}$ with $w \in S'$ are the same color. We pick
$S'=\{w:\{v_j,v_i,w\}~\textrm{is red and}~w \in S\}$ if the number
of triples $\{v_j,v_i,w\}$ with $w \in S$ is at least $\alpha |S|$
and blue otherwise. While the size of $S$ decreases now
by a much larger factor for each red edge in $\chi'$, there are
not many red edges in $\chi'$, on the other hand, we lose very little,
specifically a factor $(1-\alpha)$, for each blue edge in $\chi'$.
By picking $\alpha$ appropriately, we gain significantly over
taking $\alpha=1/2$ for our upper bound on off-diagonal hypergraph
Ramsey numbers.

Before we proceed with the proof of our upper bound on $r_3(s,n)$,
we want to discuss some other Ramsey-type numbers related to our vertex
on-line Ramsey game. One variant of Ramsey numbers which was
extensively studied in the literature (\cite{FS02}) is the {\it
size Ramsey number} $\hat r(G_1,G_2)$, which is the minimum number
of edges of a graph whose every red-blue edge-coloring contains either a
red $G_1$ or a blue $G_2$. Clearly, $\tilde{r}(k,\ell) \leq
\hat{r}(K_k,K_l)$ since builder can choose to pick the edges of a
graph which gives the size Ramsey number for $(K_k,K_l)$.
Unfortunately, it is not difficult to show that
$\hat{r}(K_k,K_l)={r(k,l) \choose 2}$ and therefore we cannot
obtain any improvement using these numbers. Another on-line Ramsey
game which is quite close to ours, was studied in \cite{KR05}. In
this game, there are two players, builder and painter, who move on
the originally empty graph with an unbounded number of vertices.
At each step, builder draws a new edge and painter has to color it
either red or blue immediately. The {\it edge on-line Ramsey
number} $\bar{r}(k,l)$ is then defined as the minimum number of
edges that builder has to draw in order to force painter to create
either a red $K_k$ or a blue $K_l$. A randomized version of the
edge on-line Ramsey game was studied in \cite{FKRRT03}. The
authors of \cite{KR05} proved an upper bound for $\bar{r}(k,l)$
which is similar to our Lemma \ref{gamelemma}. A careful reading
of their paper shows that builder first exposes edges from the
first vertex to all future vertices, and so on. Thus, this builder
strategy cannot be implemented when proving upper bounds for
hypergraph Ramsey numbers. Moreover, it is not clear how to use
the edge on-line Ramsey game to get an improvement on hypergraph
Ramsey numbers. Lemma \ref{gamelemma} is therefore essential for
our proof giving new upper bounds for hypergraph Ramsey numbers.

Using the ideas discussed above, we next prove an upper bound on
$r_3(s,n)$ which involves some parameters of the vertex on-line Ramsey
game.

\begin{theorem}\label{thmoffab}
Suppose in the vertex on-line Ramsey game that builder has a
strategy which ensures a red $K_{s-1}$ or a blue $K_{n-1}$ using
at most $v$ vertices, $r$ red edges, and in total $m$ edges. Then,
for any $0 <\alpha \leq 1/2$, the Ramsey number $r_3(s,n)$
satisfies
\begin{equation} \label{eq}r_3(s,n) \leq (v+1)\alpha^{-r}(1-\alpha)^{r-m}.
\end{equation}
\end{theorem}
\begin{proof}
Let $N=(v+1)\alpha^{-r}(1-\alpha)^{r-m}$ and consider a red-blue
coloring $\chi$ of the triples of the set $[N]$. We wish to show
that the coloring $\chi$ must contain a red set of size $s$ or a
blue set of size $n$.

We greedily construct a set of vertices $\{v_1, \cdots, v_h\}$ and
a graph $\Gamma$ on these vertices with at most $v$ vertices, at
most $r$ red edges, and at most $m$  total edges across them such
that for any edge $e = v_i v_j$, $i<j$ in $\Gamma$, the color of any 3-edge
$\{v_i, v_j, v_k\}$ with $k
> j$ is the same, say $\chi'(e)$. Moreover, this graph will contain
either a red $K_{s-1}$ or a blue $K_{n-1}$, which one can easily
see will define a red set of size $s$ or a
blue set of size $n$..

We begin the construction of this set of vertices by arbitrarily
first choosing vertices $v_1 \in [N]$ and setting $S_1=[N]
\setminus \{v_1\}$. Given a set of vertices $\{v_1, \cdots,
v_a\}$, we have a set $S_a$ such that for each edge $e = v_i v_j$ of
$\Gamma$ with $i, j \leq a$, the color of the
3-edge $\{v_i, v_j, w\}$ is the same for every $w$ in $S_a$.

Now let $v_{a+1}$ be a vertex in $S_a$. We play the vertex on-line
Ramsey game, so that builder chooses the edges to be drawn
according to his strategy. Painter then colors these edges. For
the first edge $e_1$ chosen, painter looks at all triples
containing this edge and a vertex from $S_a \setminus
\{v_{a+1}\}$. The 2-edge is colored red in $\chi'$ if it there are
more than $\alpha(|S_a|-1)$ such triples that are red and blue
otherwise. This defines a new subset $S_{a,1}$, which are all
vertices in $S_a \setminus \{v_{a+1}\}$ such that together with
edge $e_1$ form a triple of color $\chi'(e_1)$. For the next drawn
edge $e_2$ we color it red if there more than
$\alpha|S_{a,1}|$ red triples containing it and a vertex from
$S_{a,1}$ and blue otherwise. This will define an $S_{a,2}$ and so
forth. After we have added all edges from $v_{a+1}$, the remaining set will be
$S_{a+1}$. Let $m_a$ be the number of edges $e=v_iv_a$ with $i<a$
in $\Gamma$ and $r_a$ be the number of such edges that are red.

We now show by induction that $$|S_a| \geq
(v+1-a)\alpha^{-r+\sum_{i=1}^a r_i}(1-\alpha)^{r-m+\sum_{i=1}^a
m_i-r_i}.$$ For the base case $a=1$, we have
$$|S_1|=N-1=(v+1)\alpha^{-r}(1-\alpha)^{r-m}-1 \geq
v\alpha^{-r}(1-\alpha)^{r-m}.$$ Suppose we have proved the desired
inequality for $a$. When we draw a vertex $v_{a+1}$, the size
of our set $S_a$ decreases by $1$. Each time we
 draw an edge from $v_{a+1}$ we have the size of our set $S$ goes down by a factor $\alpha$ or
 $1-\alpha$. Therefore, \begin{eqnarray*} |S_{a+1}| & \geq &
 \alpha^{r_i}(1-\alpha)^{m_i-r_i}(|S_a|-1) \geq
\alpha^{r_i}(1-\alpha)^{m_i-r_i}|S_a|-1 \\ & \geq & (v+1-a)
 \alpha^{-r+\sum_{i=1}^{a+1} r_i}(1-\alpha)^{r-m+\sum_{i=1}^{a+1}
m_i-r_i}-1 \\ & \geq & ((v+1)-(a+1))\alpha^{-r+\sum_{i=1}^{a+1}
r_i}(1-\alpha)^{r-m+\sum_{i=1}^{a+1} m_i-r_i}.\end{eqnarray*}

By our assumption on the vertex on-line Ramsey game, by the time we will construct graph
$\Gamma$ containing either a read $K_{s-1}$ or a blue $K_{n-1}$, this graph will have
$a \leq v$ vertices, at most $r$ red edges, and at most $m$ total edges.
Therefore at this time we have
$$|S_a| \geq (v+1-a) \alpha^{-r+\sum_{i=1}^a
r_i}(1-\alpha)^{r-m+\sum_{i=1}^a m_i-r_i} \geq 1,$$
i.e.,  $S_a$ is
not empty. Thus a vertex from $S_a$ together with the red $K_{s-1}$
or blue $K_{n-1}$ in edge-coloring $\chi'$ of $\Gamma$ make either
a red set of size $s$ or a blue set of size $n$ in coloring
$\chi$, completing the proof.
\end{proof}

\begin{lemma}\label{gamelemma}
In the vertex on-line Ramsey game builder has a strategy which
ensures a red $K_{s}$ or a blue $K_{n}$ using at most
$\binom{s+n-2}{s-1}$ vertices, $(s-2)\binom{s+n-2}{s-1}+1$ red
edges, and $(s+n-4)\binom{s+n-2}{s-1}+1$ total edges. In
particular,
\[\tilde{r}(s,n) \leq (s+n-4) \binom{s+n-2}{s-1}+1.\]
\end{lemma}
\begin{proof}
We are going to define a set of vertices labeled by strings and the
associated set of edges to be drawn during the game as follows. The
first vertex exposed, will be labeled as $w_{\emptyset}$. Every
other vertex which we expose during the game will be connected by
edge to $w_{\emptyset}$. Recall that immediately after the edge is
exposed it is colored by painter. The first vertex which is
connected to $w_{\emptyset}$ by red (blue) edge is labeled $w_R$
($w_B$). Successively, we connect vertex $v$ to $w_R$ or $w_B$ if
and only if this vertex is already connected to $w_{\emptyset}$ by a
red or respectively blue edge.

More generally, if we have defined $w_{a_1 a_2 \cdots a_p}$ with
each $a_i = R$ or $B$ and $v$ is the first exposed vertex which is
connected to $w_{a_1 \cdots a_j}$ in color $a_{j+1}$ for each $j=
0, \cdots, p$, we label $v$ as $w_{a_1 \cdots a_{p+1}}$.
(When $j=0$, $w_{a_1 \cdots a_j}=w_{\emptyset}$.)
The only successively chosen vertices which we join to $w_{a_1 \cdots
a_{p+1}}$ by an edge will be those edges $v$ which are also joined
to $w_{a_1 \cdots a_j}$ in color $a_{j+1}$ for each $j= 0, \cdots,
p$.

Suppose now that we have exposed $\binom{s+n-2}{s-1}=\binom{s+n-3}{s-2}+
\binom{s+n-3}{s-1}$ vertices in
total. Since $w_{\emptyset}$ is connected to all vertices, its degree is
$\binom{s+n-2}{s-1}-1$. Thus $w_{\emptyset}$ is
connected either to
$\binom{s+n-3}{s-2}$ vertices in red or $\binom{s+n-3}{s-1}$
vertices in blue. If the former holds we look at the neighbors
of $w_R$, which is all vertices which labeled by string with first letter $R$. Otherwise we look at
neighbors of $w_B$. Suppose now that we are looking at the neighbors of
$w_{a_1 \cdots a_p}$, where $r$ of the $a_i$ are red and $b$ of them are
blue. Then, by our  construction, $w_{a_1 \cdots a_p}$ will have been
joined to $\binom{s+n-r-b-2}{s-r-1} - 1$ vertices. Now, either
$w_{a_1 \cdots a_p}$ is joined to $\binom{s+n-r-b-3}{s-r-2}$
vertices in red or $\binom{s+n-r-b-3}{s-r-1}$ vertices in blue. In
the first case we look at $w_{a_1 \cdots a_p R}$ and its neighbors and in the second
case at $w_{a_1 \cdots a_p B}$ and its neighbors. By the time we reach a string of
length $s+n-3$, we will have either $s-1$ reds or $n-1$ blues. If
$s-1$ of the $a_i$, say $a_{j_1+1}, \cdots, a_{j_{s-1}+1}$, are $R$,
then we know that the collection of vertices
$w_{a_1 \cdots a_{j_1}}, \cdots, w_{a_1 \cdots a_{j_{s-1}}},
w_{a_1 \cdots a_{j_{s-1}+1}}$ forms
a red clique of size $s$. Similarly, were $n-1$ of the $a_i$ blue,
we would have a blue clique of size $n$.

All that remains to do is to estimate how many edges builder
draws. Look on the vertices in the order they were exposed.
Clearly, for every vertex we can only look on the edges connecting it to preceeding vertices.
Notice that a vertex $w_{a_1 \cdots a_p}$ is
adjacent to precisely
$p$ vertices which were exposed before it. Moreover the number of red edges
connecting $w_{a_1 \cdots a_p}$ to vertices before it is
precisely the number of $a_i$ which are $R$. Since all
but the last vertex are labeled by strings of length at most $s+n-4$, we have at most
$(s+n-4)\binom{s+n-2}{s-1}+1$ total edges.
Similarly, all but the last vertex have at most $s-2$ symbols $R$ in their string, which shows that
the number of edges colored red during the game is at most
$(s-2)\binom{s+n-2}{s-1}+1$.
\end{proof}

The following result implies (\ref{mainlowerbound}).

\begin{corollary}
The Ramsey number $r_3(s,n)$ with $4 \leq s \leq n$ satisfies
\begin{equation} r_3(s,n) \leq
2^{\frac{(s-3)}{(s-2)!}(s+n)^{s-2}\log_2 (64n/s)}.
\end{equation}
\end{corollary}

\noindent
{\bf Proof.} \hspace{1mm}
By Lemma \ref{gamelemma}, in the vertex on-line Ramsey game
builder has a strategy which ensures a red $K_{s-1}$ or a blue
$K_{n-1}$ using at most $v=\binom{s+n-4}{s-2}$ vertices,
$r=(s-3)\binom{s+n-4}{s-2}+1$ red edges, and
$m=(s+n-6)\binom{s+n-4}{s-2}+1$ total edges. To minimize the
function $\alpha^{-r}(1-\alpha)^{r-m}$, one should take
$\alpha=\frac{r}{m}$.
Note that $m/r \leq (s+n-6)/(s-3) \leq 4n/s$, $v < r \leq m/2$
and $r \leq \frac{(s-3)}{(s-2)!}(s+n)^{s-2}$. Hence, the Ramsey
number $r_3(s,n)$ satisfies
\begin{eqnarray*}
\hspace{2.4cm}
r_3(s,n) & \leq & (v+1)(m/r)^{r}(1-r/m)^{r-m} \leq
(v+1)(4n/s)^r(1-r/m)^{-m} \\
&\leq& r(4n/s)^r \big(1+2r/m\big)^m \leq r(4e^2n/s)^r <(64n/s)^r\\
& \leq & 2^{\frac{(s-3)}{(s-2)!}(s+n)^{s-2}\log_2 (64n/s)}.
\hspace{7.6cm} \Box
\end{eqnarray*}

\noindent Also, taking $\alpha=1/2$ in Theorem \ref{thmoffab}, it
is worth noting that in the diagonal case our results easily imply
the following theorem, which improves upon the bound $r_3(k,k)
\leq 2^{2^{4k}}$ due to Erd\H{o}s and Rado.
\begin{theorem}
\[\log_2 \log_2 r_3(k,k) \leq (2+o(1))k.\]
\end{theorem}

Our methods can also be used to study Ramsey numbers of
non-complete hypergraphs. To illustrate this, we will obtain a
lower bound on $f_3(N,4,3)$, slightly improving a result of
Erd\H{o}s and Hajnal. Let $K_t^{(3)}$ denote the complete
$3$-uniform hypergraph with $t$ vertices, and $K_t^{(3)}\setminus
e$ denote the $3$-uniform hypergraph with $t$ vertices formed by
removing one triple. For $k$-uniform hypergraphs $H$ and $G$, the
Ramsey number $r(H,G)$ is the minimum $N$ such that every red-blue
coloring of ${[N] \choose k}$ contains either a red copy of $H$ or
a blue copy of $G$. Note that an upper bound on $f_3(N,4,3)$ is
equivalent to a lower bound on the Ramsey number
$r\big(K_4^{(3)}\setminus e,K^{(3)}_n\big)$ because they are
inverse functions of each other. Erd\H{o}s and Hajnal \cite{EH72}
proved the following bounds: $$\frac{1}{2}\frac{\log_2 N}{\log_2
\log_2 N} \leq f_3(N,4,3) \leq (2\log_2 N)+1.$$ The upper bound
follows from the same coloring (discussed in the introduction)
based on tournaments which gives a lower bound on $r_3(4,n)$. We
will use our approach to improve the lower bound by an asymptotic
factor of $2$.
\begin{proposition}
We have $r \big(K_4^{(3)} \setminus e,K^{(3)}_n \big) \leq
(2en)^n$.
\end{proposition}
{\bf Sketch of proof.}\hspace{1mm} We apply the exact same proof
technique as we did for Theorem \ref{thmoffab} except that we will
expose all edges. We have a coloring of the complete $3$-uniform
hypergraph with $N$ vertices which neither contains a red
$K_4^{(3)}\setminus e$ nor a blue set of size $n$. Note that the
coloring $\chi'$ of the edges of the complete graph with vertex set
$V=\{v_1,\ldots,v_{h-1}\}$ we get in the proof does not contain a
pair of monochromatic red edges $(v_j,v_i)$ and $(v_j,v_k)$ with $1
\leq j < i<k<h$ or $1 \leq i <j<k<h$, otherwise $v_i,v_j,v_k,v_h$
are the vertices of a red $K_4^{(3)}\setminus e$. Therefore, the red
graph in the coloring $\chi'$ is just a disjoint union of stars. Let
$m$ be the number of edges in the red graph. Note that disjoint
union of stars with $m$ edges has an independent set of size $m$ and
forms a bipartite graph. Therefore the red graph has an independent
set of size at least $\max\{m, (h-1)/2\}$. Such an independent set
in the red graph is a clique in the blue graph in the coloring
$\chi'$, and together with $v_h$ make a blue complete $3$-uniform
hypergraph in the coloring $\chi$. This, gives us inequalities $m+1
<n$ and $(h-1)/2+1<n$. With hindsight, we pick $\alpha=1/(2n)$. By
Theorem \ref{thmoffab}, this implies that
$$r(K_4^{(3)}
\setminus e,n) \leq (1+h)\alpha^{-m}(1-\alpha)^{m-{h-1 \choose
2}}\leq
(2n-1)\big(2n\big)^{n-2}\left(1-\frac{1}{2n}\right)^{-h^2/2}
\leq
(2en)^{n},$$ where we use that $3 \leq h \leq 2n-2, m \leq n-2$ and that $(1-1/x)^{1-x} \leq e$
for $x>1$. \qed

Theorem \ref{lowerbound} shows that $\log r_3(4,n) >cn \log n$ for
an absolute constant $c$. It would be also nice to give a similar
lower bound (if it is true) for $r\big(K_4^{(3)}\setminus
e,K_n^{(3)}\big)$ since then we would know that $\log
r\big(K_4^{(3)} \setminus e,K_n^{(3)}\big)$ has order of $n
\log n$.

\section{A lower bound construction}

The purpose of this section is to prove Theorem \ref{lowerbound}
which gives a new lower bound on $r_3(s,n)$. To do this, we need
to recall an estimate for graph Ramsey numbers. As we already
mentioned in (\ref{eq2}), for sufficiently large $n$ and fixed
$s$, $r(s,n)>c\left(n/\log n\right)^{(s+1)/2}>n^{3/2}$. Also, for
all $4 \leq s \leq n$ and $n$ sufficiently large, one can easily
show that $r(s,n)>(\frac{n+s}{s})^{s/3}$. (This is actually not
the best lower bound for $r(s,n)$ but it is enough for our
purposes.) Indeed, if $s=3$, this bound is trivial. For $s \geq
4$, consider a random red-blue edge-coloring of the complete graph
on $N=(\frac{n+s}{s})^{s/3}$ vertices in which each edge is red
with probability $p=\big(\frac{s}{n+s}\big)^{0.9}$. It is easy to
check that the expected number of monochromatic red $s$-cliques
and blue $n$-cliques in this coloring is ${N \choose s}p^{{s
\choose 2}}+{N \choose n}(1-p)^{{n \choose 2}}<1$. These estimates
together with the next theorem clearly imply Theorem
\ref{lowerbound}.

\begin{theorem}
For all sufficiently large $n$ and $4 \leq s \leq n$,
$$r_3(s,n)>\big(r(s-1,n/4)-1\big)^{n/24}.$$
\end{theorem}
\begin{proof}
Let $\ell=n/4$, $r=r(s-1,\ell)-1$, $N=r^{n/24}$, and $c_1:{[r]
\choose 2} \rightarrow \{\textrm{red, blue}\}$ be a red-blue
edge-coloring of the complete graph on $[r]$ with no red clique of
size $s-1$ and no blue clique of size $\ell$. Consider a coloring
$c_2:{[N] \choose 2} \rightarrow [r]$ picked uniformly at random
from all $r$-colorings of ${[N] \choose 2}$, i.e., each edge has
probability $\frac{1}{r}$ of being a particular color independent
of all other edges. Using the auxiliary colorings $c_1$ and $c_2$,
we define the red-blue coloring $c:{[N] \choose 3} \rightarrow
\{\textrm{red, blue}\}$ where the color of a triple $\{a,b,c\}$
with $a<b<c$ is $c_1\big(c_2(a,b),c_2(a,c)\big)$ if $c_2(a,b) \not
= c_2(a,c)$ and is blue if $c_2(a,b)=c_2(a,c)$. We next show that
in coloring $c$ there is no red set of size $s$ and with positive
probability no blue set of size $n$, which implies the theorem.

First, suppose that the coloring $c$ contains a red set
$\{u_1,\ldots,u_s\}$ of size $s$ with $u_1 < \ldots < u_s$. Then all
the colors $c_2(u_1,u_j)$ with $2 \leq j \leq s$ are distinct and
form a red clique of size $s-1$ in $c_1$, a contradiction.

Next, we estimate the expected number of blue cliques of size $n$ in
coloring $c$. Let $\{v_1,\ldots,v_n\}$ with $v_1<\ldots<v_n$ be a
set of $n$ vertices. Fix for now $1 \leq i \leq n$. If all triples
$\{v_i,v_j,v_k\}$ with $i<j<k$ are blue, then the distinct colors
among the colors $c_2(v_i,v_j)$ for $i<j \leq n$ must form a blue
clique in coloring $c_1$. Therefore the number of distinct colors
$c_2(v_i,v_j)$ with $i<j \leq n$ is less than $\ell$. Every such
subset of distinct colors is contained in at least one of the ${r
\choose \ell-1}$ subsets of $[r]$ of size $\ell-1$.  If we fix a set
of $\ell-1$ colors, the probability that each of the colors
$c_2(v_i,v_j)$ with $i<j \leq n$ is one of these $\ell-1$ colors is
$\left(\frac{\ell-1}{r}\right)^{n-i}$. Therefore the expected number
of blue cliques of size $n$ in coloring $c$ is at most

\begin{eqnarray*} {N \choose n}\prod_{i=1}^n {r \choose \ell-1}
\left(\frac{\ell-1}{r}\right)^{n-i} & \leq & N^n {r \choose
\ell-1}^n\left(\frac{\ell-1}{r}\right)^{{n \choose 2}}\leq N^n
\left(\frac{e r}{\ell-1}\right)^{(\ell-1)
n}\left(\frac{\ell-1}{r}\right)^{{n \choose 2}} \\ & = &
\left(Ne^{\ell-1}\left(\frac{\ell-1}{r}\right)^{\frac{n-1}{2}-(\ell-1)}
\right)^n  <
\left(Ne^{\ell-1}\left(r^{-1/3}\right)^{\frac{n-1}{2}-(\ell-1)}
\right)^n \\ & < & \left(N^2\left(r^{-1/3}\right)^{n/4} \right)^n
=1,\end{eqnarray*} where we use that $\ell-1<r^{2/3}$, $\ell=n/4$,
and $N=r^{n/24} =r^{\ell/6}>e^{\ell}$. Hence, there is a coloring
$c$ with no red set of size $s$ and no blue set of size $n$.
\end{proof}

An additional feature of our new lower bound on $r_3(s,n)$ is that
it increases continuously with growth of $s$ and for $s=n$
coincides with the bound $r_3(n,n) \geq 2^{cn^2}$, which was given
by Erd\H{o}s, Hajnal, and Rado \cite{EHR65}. For example, for
$n^{1/2} \ll s \ll n$, the previously best known bound for
$r_3(s,n)$ was essentially $r_3(s,n) \geq r_3(s,s) \geq 2^{cs^2}$.

\section{Bounding $r_3(n,n,n)$}\label{sectionthreecolor}

We now prove the lower bound, $r_3 (n,n,n) \geq 2^{n^{c \log n}}$,
mentioned in the introduction. Though our method follows the
stepping-up tradition of Erd\H{o}s and Hajnal, it is curious to note
that their own best lower bound on the problem, $r_3(n,n,n) \geq
2^{c n^2 \log^2 n}$, is not proven in this manner. In Erd\H{o}s and
Hajnal's proof that the $r_3(n,n,n,n)>2^{2^{cn}}$, they use the
stepping up lemma starting from a $2$-coloring of a complete graph
with $r(n-1,n-1)-1$ vertices not containing a monochromatic clique
of size $n-1$ to obtain a $4$-coloring of the triples of a set of
size $2^{r(n-1,n-1)-1}$ without a monochromatic set of size $n$. Our
proof that $r_{3}(n,n,n)>2^{n^{c\log n}}$ is also based on the
stepping-up lemma, using essentially the following idea. We start
with a $2$-coloring of the complete graph on $r(\log_2 n, n-1)-1$
vertices which contains neither a monochromatic red clique of size
$\log_2 n$ nor a monochromatic blue clique of size $n-1$. Then we
obtain a $4$-coloring of the triples of a set of size $2^{r(\log_2
n, n-1)-1}\geq 2^{n^{c\log n}}$ as in the Erd\H{o}s-Hajnal proof.
Next we combine two of the four color classes to obtain a
$3$-coloring of the triples. Finally, we carefully analyze this
$3$-coloring to show that it does not contain a monochromatic set of
size $n$.

\begin{theorem} \label{steppingup}
\[r_3(n,n,n) > 2^{r(\log_2 n, n-1) - 1}.\]
\end{theorem}

\begin{proof}
Let $G$ be a graph on $m=r(\log_2 n, n-1) - 1$ vertices which
contains neither a clique of size $n-1$ nor an independent set of
size $\log_2 n$ and let $\bar G$ be the complement of $G$. We are
going to consider the complete $3$-uniform hypergraph $H$ on the set
\[T = \{(\gamma_1, \cdots, \gamma_m) : \gamma_i = 0 \mbox{ or }
1\}.\]

If $\epsilon = (\gamma_1, \cdots, \gamma_m)$, $\epsilon' =
(\gamma'_1, \cdots, \gamma'_m)$ and $\epsilon \neq \epsilon'$,
define
\[\delta(\epsilon, \epsilon') = \max\{i : \gamma_i \neq
\gamma'_i\},\] that is, $\delta(\epsilon, \epsilon')$ is the
largest coordinate at which they differ. Given this, we can define
an ordering on $T$, saying that
\[\epsilon < \epsilon' \mbox{ if } \gamma_i = 0, \gamma'_i = 1,\]
\[\epsilon' < \epsilon \mbox{ if } \gamma_i = 1, \gamma'_i = 0.\]
Equivalently, associate to any $\epsilon$ the number $b(\epsilon)
= \sum_{i=1}^m \gamma_i 2^{i-1}$. The ordering then says simply
that $\epsilon < \epsilon'$ iff $b(\epsilon) < b(\epsilon')$.

We will further need the following two properties of the
function $\delta$ which one can easily prove.\\

(a) If $\epsilon_1 < \epsilon_2 < \epsilon_3$, then
$\delta(\epsilon_1, \epsilon_2) \neq \delta(\epsilon_2,
\epsilon_3)$ and

(b) if $\epsilon_1 < \epsilon_2 < \cdots < \epsilon_p$, then
$\delta (\epsilon_1, \epsilon_p) = \max_{1 \leq i \leq p-1}
\delta(\epsilon_i, \epsilon_{i+1})$. \\
In particular, these properties imply that there is a unique index $i$ which achieves
maximum of $\delta(\epsilon_i, \epsilon_{i+1})$. Indeed suppose that there are indices $i<i'$ such that
$$\ell =\delta(\epsilon_i, \epsilon_{i+1})=\delta(\epsilon_{i'}, \epsilon_{i'+1})=\max_{1 \leq j \leq
p-1} \delta(\epsilon_j, \epsilon_{j+1}).$$
Then, by property (b) we also have that
$\ell=\delta(\epsilon_i, \epsilon_{i'})=\delta(\epsilon_{i'}, \epsilon_{i'+1})$. This contradicts
property (a) since $\epsilon_{i}<\epsilon_{i'}<\epsilon_{i'+1}$.

We are now ready to color the complete $3$-uniform hypergraph $H$
on the set $T$. If $\epsilon_1 < \epsilon_2 < \epsilon_3$, let
$\delta_1 = \delta(\epsilon_1, \epsilon_2)$ and $\delta_2 =
\delta(\epsilon_2, \epsilon_3)$. Note that, by property (a) above,
$\delta_1$ and $\delta_2$ are not equal. Color the edge
$\{\epsilon_1, \epsilon_2, \epsilon_3\}$ as follows:\\

$C_1$, if $(\delta_1, \delta_2) \in e(G)$ and $\delta_1 <
\delta_2$;

$C_2$, if $(\delta_1, \delta_2) \in e(G)$ and $\delta_1 >
\delta_2$;

$C_3$, if $(\delta_1, \delta_2) \not\in e(G)$, i.e.,
it is an edge in $\bar G$.\\

Suppose that $C_1$ contains a clique $\{\epsilon_1, \cdots,
\epsilon_n\}_<$ of size $n$. For $1 \leq i \leq n-1$, let $\delta_i
= \delta(\epsilon_i, \epsilon_{i+1})$. Note that the $\delta_i$ form
a monotonically increasing sequence, that is $\delta_1 < \delta_2 <
\cdots < \delta_{n-1}$. Also, note that since, for any $1 \leq i < j
\leq n-1$, $\{\epsilon_i, \epsilon_{i+1}, \epsilon_{j+1}\} \in C_1$,
we have, by property (b) above, that $\delta(\epsilon_{i+1},
\epsilon_{j+1}) = \delta_j$, and thus $\{\delta_i, \delta_j\} \in
e(G)$. Therefore, the set $\{\delta_1, \cdots, \delta_{n-1}\}$ must
form a clique of size $n-1$ in $G$. But we have chosen $G$ so as not
to contain such a clique, so we have a contradiction. A similar
argument shows that $C_2$ also cannot contain a clique of size $n$.

For $C_3$, assume again that we have a monochromatic clique
$\{\epsilon_1, \cdots, \epsilon_n\}_<$ of size $n$, and, for $1 \leq
i \leq n-1$, let $\delta_i = \delta(\epsilon_i, \epsilon_{i+1})$.
Not only can we no longer guarantee that these $\delta_i$ form a
monotonic sequence, but we can no longer guarantee that they are
distinct. Suppose that there are $d$ distinct values of $\delta$,
given by $\{\Delta_1, \cdots, \Delta_d\}$, where $\Delta_1 > \cdots
> \Delta_d$. We will consider the subgraph of $\bar G$ induced by
this vertices. Note that, by definition of the coloring $C_3$, the
vertices $\Delta_i$ and $\Delta_j$ are adjacent in $\bar G$ if there
exists $\epsilon_r < \epsilon_s < \epsilon_t$ with
$\Delta_i=\delta(\epsilon_r, \epsilon_s)$ and $ \Delta_j =
\delta(\epsilon_s,\epsilon_t)$. We show that this set necessarily
has a complete subgraph on $\log_2 n$ vertices, contradicting our
assumptions on $\bar G$.

Since $\Delta_1$ is the largest of the $\Delta_j$, there is a unique
index $i_1$ such that $\Delta_1=\delta_{i_1}$. Note that $\Delta_1$ is adjacent in $\bar G$ to
all $\Delta_j, j>1$. Indeed, every such $\Delta_j=\delta(\epsilon_{i'}, \epsilon_{i'+1})$  for some
index $i'\not =i_1$ and suppose $i_1<i'$ (the other case is similar). Then
$\epsilon_{i_1} < \epsilon_{i'} < \epsilon_{i'+1}$. Also, by property (b) and maximality
of $\Delta_1$, we have that $\Delta_1=\delta(\epsilon_{i_1}, \epsilon_{i_1+1})=\delta(\epsilon_{i_1},
\epsilon_{i'})$, and therefore it is connected to $\Delta_j$ in $\bar G$.
Now either there are $(n-2)/2=n/2-1$
values of $j$ greater than $i_1$ or less than $i_1$. Let $V_1$ be
the larger of these two intervals.

Suppose, inductively, that one has been given an interval $V_{j-1}$ in
$[n-1]$. Look at the set $\{\delta_a | a \in V_{j-1}\}$. One of
these $\delta$, say $\delta_{i_{j}}$, will be the largest
and as we explain above, will be connected to every other $\delta_a$ with $a \in
V_{j-1}$. There are at least $(|V_{j-1}| - 1)/2$ indices in either $\{a \in
V_{j-1} | a < i_j\}$ or $\{a \in V_{j-1} | a > i_j\}$. Let $V_{j}$ be
the larger of these two intervals, so in particular $|V_{j}| \geq (|V_{j-1}| -
1)/2$. By induction, it is easy to show that $|V_j| \geq \frac{n}{2^j} - 1$. Therefore,
for $j \leq \log_2 n-1$, $|V_j| \geq 1$ and, hence, the set
$\delta_{i_1}, \cdots, \delta_{i_{\log_2 n}}$ forms a
clique in $\bar G$, as required. This contradicts the fact that $G$ has no independent set of this size
and completes the proof.
\end{proof}

As discussed in the beginning of Section 3, the probabilistic method
demonstrates that, for $s \leq n$, $r(s,n) \geq
\big(\frac{n+s}{s}\big)^{c's}$. Substituting this bound with
$s=\log_2 n$ into Theorem \ref{steppingup} implies the desired
result $r_3(n,n,n) \geq 2^{n^{c \log n}}$.

\section{A hypergraph problem of Erd\H{o}s and
Hajnal}\label{sectionnext}

In this section we determine the function $h_1^{(3)}(s)$ for
infinitely many values of $s$ and find a small interval containing
$h_1^{(3)}(s)$ for all values of $s$. Recall that $f_3(N,s,t)$ is
the largest integer $n$ for which every red-blue coloring of ${[N]
\choose 3}$ has a red $n$-element set or a set of size $s$ with at
least $t$ blue triples. Also recall that $h_1^{(3)}(s)$ is the
least $t$ for which $f_3(N,s,t)$ stops growing like a power of $N$
and starts growing like a power of $\log N$, i.e.,
$f_3(N,s,h_1^{(3)}(s)-1)>N^{c_1}$ but $f_3(N,s,h_1^{(3)}(s))<(\log
N)^{c_2}$.

Consider the minimal family $\mathcal{F}$ of $3$-uniform
hypergraphs defined as follows. The empty hypergraphs on $1$ and
$2$ vertices and an edge are elements of $\mathcal{F}$. If $H,G
\in \mathcal{F}$ and $v$ is a vertex of $H$, then the following
$3$-uniform hypergraph $H(G,v)$ is in $\mathcal{F}$ as well. The
vertex set of $H(G,v)$ is $(V(H) \setminus \{v\}) \cup V(G)$ and
its edges consist of the edges of $G$, the edges of $H$ not
containing $v$, and all triples $\{a,b,c\}$ with $a,b \in H$ and
$c \in G$ for which $\{a,b,v\}$ is an edge of $H$. Erd\H{o}s and
Hajnal showed that for every hypergraph $H \in \mathcal{F}$ on $s$
vertices, every red-blue coloring of the triples of a set of size
$N$ has either a red copy of $H$ or a blue set of size
$N^{\epsilon_s}$. This can be shown by induction on $s$ using the
following claim which can be proved using a simple counting
argument. If $H,G \in \mathcal{F}$, then any red-blue coloring of
the triples of a set of size $N$ without a blue set of size
$N^{\epsilon_1}$ has $N^{\delta}$ copies of $H$ all sharing the
same copy of $H \setminus v$. In these $N^{\delta}$ vertices, we
either get a red copy of $G$ which together with the copy of $H
\setminus v$ make a copy of $H(G,v)$ or a blue set of size
$(N^{\delta})^{\epsilon_2}$. By choosing
$\epsilon=\min(\epsilon_1,\delta\epsilon_{2})$, we get the desired
result.

Let $g_1^{(3)}(s)$ be the maximum number of edges in a hypergraph
in $\mathcal{F}$ with $s$ vertices. One can check that every
hypergraph $H \in \mathcal{F}$ on $s$ vertices has the following
structure. Its vertex set can be partitioned into three parts $A,
B, C$ (one of which might be empty) such that all triples
intersecting $A$, $B$, and $C$ are edges of $H$ and subhypergraphs
induced by sets $A$, $B$, and $C$ are also members of
$\mathcal{F}$. This implies that the function $g_1^{(3)}(s)$ can
also be defined recursively. Put $g_1^{(3)}(1)=g_1^{(3)}(2)=0$.
Assume that $g_1^{(3)}(m)$ has already been defined for all $m<s$.
Then
$$g_1^{(3)}(s)=\max_{a+b+c=s}g_1^{(3)}(a)+g_1^{(3)}(b)+g_1^{(3)}(c)+abc.$$
It is not difficult to see that the maximum is obtained when $a$,
$b$, and $c$ are as nearly equal as possible. It follows from the
definition of $h_1^{(3)}$ and the result in the previous paragraph
that $h_1^{(3)}(s)>g_1^{(3)}(s)$. Erd\H{o}s and Hajnal further
conjectured that this bound is tight.

\begin{conjecture}\label{confirst}
For all positive integers $s$, $h_1^{(3)}(s)=g_1^{(3)}(s)+1$.
\end{conjecture}

Consider an edge-coloring $c$ of the complete graph on $[N]$ with
colors $I,II,III$ picked uniformly at random. From this coloring,
we get a red-blue coloring $C$ of the triples from $[N]$ as
follows: if
 $a<b<c$ has $(a,b)$ color $I$, $(b,c)$ color
$II$, and $(a,c)$ color $III$, then color $\{a,b,c\}$ red,
otherwise color the triple blue. Since every complete graph of
order $q$ contains $\Theta(q^2)$ edge-disjoint triangles, the
probability that a given set of size $q$ contains only blue
triples is at most $2^{-\Theta(q^2)}$. Thus, it is straightforward
to check that with high probability, in the coloring $C$ the
largest blue set has size $O(\log N)$. Over all edge-colorings of
the complete graph on $[s]$ with colors $I,II,III$, let $F_1(s)$
denote the maximum number of triples $(a,b,c)$ with $1 \leq a<b<c
\leq s$ such that $(a,b)$ is color $I$, $(b,c)$ is color $II$, and
$(a,c)$ is color $III$. Note that in the coloring $C$ we
constructed above whose largest blue set has size $O(\log N)$,
every set of size $s$ has at most $F_1(s)$ red triples. Therefore,
by definition, $h_1^{(3)}\leq F_1(s)+1$. Since also
$h_1^{(3)}(s)>g_1^{(3)}(s)$, it implies that $F_1(s) \geq
 g_1^{(3)}(s)$. Erd\H{o}s and Hajnal conjectured that these two functions are
actually equal, which would imply
$h_1^{(3)}(s)=F_1(s)+1=g_1^{(3)}(s)+1$ and hence Conjecture
\ref{confirst}.

 \begin{conjecture}\label{consecond}
For all positive integers $s$, $F_1(s)=g_1^{(3)}(s)$.
\end{conjecture}
Erd\H{o}s and Hajnal verified Conjectures \ref{confirst} and
\ref{consecond} for $s \leq 9$.

To attack these conjectures, we use a new function which was not
considered in \cite{EH72}. Let $T(s)$ be the maximum number of
directed triangles in all tournaments on $s$ vertices. It is an
exercise (see, e.g., \cite{L07}) to check that every tournament
with $n$ vertices of outdegrees $d_1,\ldots,d_n$ has exactly ${n
\choose 3}-\sum_{i=1}^n {d_i \choose 2}$ cyclic triangles. This
number is maximized when all the $d_i$ are as equal as possible,
that is, if $n$ is odd, $d_i=\frac{n-1}{2}$ for all $i$ and, if
$n$ is even, half of the $d_i$ are $\frac{n-2}{2}$ and the other
half are $\frac{n}{2}$. It is easy to see that there is a
tournament with such outdegrees, and therefore we have the
following formula for $T(s)$:
\begin{equation}
\label{eqtwo} T(s)=\left\{\begin{array}{ll}
{\frac{(s+1)s(s-1)}{24}}&\mbox{ if $s$ is odd}\\
{\frac{(s+2)s(s-2)}{24}}&\mbox{ if $s$ is even.}
\end{array}\right.
\end{equation}

It appears that $T(s)$ and $F_1(s)$ are closely related. Indeed,
given an edge-coloring of the complete graph on $[s]$ with colors
$I,II,III$, construct the following tournament on $[s]$. If
$(a,b)$ with $a<b$ is color $I$ or $II$, then direct the edge from
$a$ to $b$ and otherwise direct the edge from $b$ to $a$. Note
that any triple $(a,b,c)$ with $a<b<c$ and $(a,b)$ color $I$,
$(b,c)$ color $II$, and $(a,c)$ color $III$ makes a cyclic
triangle in our tournament. We therefore have $F_1(s) \leq T(s)$.
Let us summarize the inequalities we have seen so far:
\begin{equation}\label{eqthree} g_1^{(3)}(s) \leq h_1^{(3)}(s)-1 \leq F_1(s)
\leq T(s) .\end{equation}

Let $d(s)=g_1^{(3)}(s)-T(s)$. We have $d(s)=0$ if and only if all
the inequalities in (\ref{eqthree}) are equalities. We call such a
number $s$ {\it nice}. Note that Conjectures \ref{confirst} and
 \ref{consecond} necessarily hold in the case $s$ is nice.
Using this fact, we next find infinitely many values of $s$ for
which Conjectures \ref{confirst} and \ref{consecond} hold.

\begin{proposition} If $s$ is a power of $3$, then
$$g_1^{(3)}(s) = h_1^{(3)}(s)-1 = F_1(s) = T(s)=\frac{1}{4}{s+1
\choose 3}.$$
\end{proposition}

\noindent {\bf Proof.}\hspace{1mm} We easily see that $s=1$ is
nice. By induction, the proposition follows from checking that if
$s$ is odd and nice, then so is $3s$. Since, by definition,
$g_1^{(3)}(3s)=s^3+3g_1^{(3)}(s)$, we indeed have $$\hspace{1.1cm}
d(3s)=T(3s)-g_1^{(3)}(3s)=\frac{1}{4}{3s+1 \choose
3}-3g_1^{(3)}(s)-s^3=\frac{3}{4}{s+1 \choose
3}-3g_1^{(3)}(s)=3d(s). \hspace{1.1cm} \Box $$

The computation in the proof of the proposition above shows that
if $s=6x+3$ with $x$ a nonnegative integer, then $d(s)=3d(2x+1)$.
One can check the other cases of $s \pmod 6$ rather easily.

\begin{lemma}\label{lemmad}
If $x$ is a positive integer, then
\begin{eqnarray*} d(6x-2)& = & 2d(2x-1)+d(2x), \\ d(6x-1) & = &
d(2x-1)+2d(2x)+x, \\ d(6x) & = & 3d(2x), \\ d(6x+1) & = &
2d(2x)+d(2x+1)+x, \\ d(6x+2) & = & d(2x)+2d(2x+1),
\\ d(6x+3)& = & 3d(2x+1).\end{eqnarray*}
\end{lemma}
Note that from this lemma, we can easily determine which
 values of $s$ are nice. In particular, the nice
positive integers up to 100 are
$$1,2,3,4,6,8,9,10,12,18,24,26,27,28,30,36,54,72,78,80,81,82,84,90.$$

Also, from Lemma \ref{lemmad}, we can easily prove an upper bound
on $d(s)$.
\begin{proposition}
For all positive integers $s$, $d(s)=O(s \log s)$.
\end{proposition}
\begin{proof}
Let $D(s)=d(s)-cs \log s$ with $c$ a sufficiently large constant.
Using induction on $s$ and the recursive formula for $d(s)$
depending on $s \pmod 6$ in Lemma \ref{lemmad}, we get that $D(s)$
is negative for $s>1$. Indeed, assuming $s=6x+1$ with $x$ a
positive integer (the other five cases are handled similarly), we
get
\begin{eqnarray*} D(s) & = & d(6x+1)-c(6x+1)\log
(6x+1)=2d(2x)+d(2x+1)+x-c(6x+1)\log (6x+1) \\ & < &
2d(2x)+d(2x+1)-c(6x+1)\log (2x+1) <
2D(2x)+D(2x+1)<0.\end{eqnarray*}
\end{proof}

\noindent The above proposition demonstrates that $T(s)$ and
$g_1^{(3)}(s)$, which are cubic in $s$, are always fairly close
together. Therefore, using (\ref{eqtwo}), we have that
$h_1^{(3)}(s)$ always lies in an interval of length $O(s \log s)$
around $s^3/24$.

In their attempt to determine $h_1^{(3)}(s)$, Erd\H{o}s and Hajnal
consider yet another function. Consider a coloring of the edges of
the complete graph on $s$ vertices labeled $1,\ldots,s$ by two
colors I and II which maximizes the number of triangles $(a,b,c)$
with $1\leq a<b<c \leq s$ such that $(a,b)$ and $(b,c)$ has color
I, and $(a,c)$ has color II. Denote this maximum by $F_2(s)$.
Trivially, $F_2(s) \geq F_1(s)$. Erd\H{o}s and Hajnal thought that
``perhaps $F_2(s)=F_1(s)$''. As we will show, this is indeed the
case for some values of $s$, but is not true in general. For
example, it is false already for $s=5$ and $s=7$. Moreover, we
precisely determine the
 $F_2(s)$ for all values of $s$.

\begin{lemma}
For all positive integers $s$, $F_2(s)=T(s)$.
\end{lemma}
\begin{proof} We first show that $T(s) \geq F_2(s)$. Indeed, from a two
coloring with colors I and II of the edges of the complete graph
with vertices $1,\ldots,s$ we get a tournament on $s$ vertices as
follows: if $(a,b)$ with $a<b$ is color I, then orient the edge
from $a$ to $b$, otherwise $(a,b)$ is color II and orient the edge
from $b$ to $a$. Any triangle $(a,b,c)$ with $a<b<c$ with $(a,b)$
and $(b,c)$ color I and $(a,c)$ color II is a cyclic triangle in
the tournament, and the inequality $T(s) \geq F_2(s)$ follows.

We next show that actually $T(s)=F_2(s)$. Consider the two
coloring of the edges of the complete graph on $s$ vertices where
$(a,b)$ is color II if and only if $b-a$ is even. A simple
calculation shows that the number of triangles $(a,b,c)$ with
$a<b<c$ with $(a,b)$ and $(b,c)$ color I and $(a,c)$ color II in
this coloring is precisely the formula (\ref{eqtwo}) for $T(s)$.
Assume $s$ is even (the case $s$ is odd can be treated similarly).
For fixed $a$ and $c$ with $c-a$ even, the number of such
triangles containing edge $(a,c)$ is $\lfloor \frac{c-a}{2}
\rfloor$. Letting $c=a+2i$, we thus have $$F_2(s) \geq
\sum_{a=1}^s\sum_{1 \leq i \leq \lfloor \frac{s-a}{2}
\rfloor}i=\sum_{a=1}^s {\lfloor \frac{s-a}{2} \rfloor+1 \choose
2}=\sum_{j=1}^{s/2}2{j \choose 2}=2{\frac{s}{2}+1 \choose
3}=T(s),$$ and hence $F_2(s)=T(s)$.
\end{proof}
\section{Odds and ends}\label{oddsandends}

\subsection{Polynomial versus Exponential Ramsey numbers}

As we discussed in Section \ref{offdiagonalsection}, the Ramsey
number of $K_4^{(3)}\setminus e$ versus $K_n^{(3)}$ is at least
exponential in $n$. The hypergraph $K_4^{(3)}\setminus e$ is a
special case of the following construction. Given an arbitrary
graph $G$, let $H_G$ be the $3$-uniform hypergraph whose vertices
are the vertices of $G$ plus an auxiliary vertex $v$. The edges of
$H_G$ are all triples obtained by taking the union of an edge of
$G$ with vertex $v$. For example, by taking $G$ to be the
triangle, we obtain $K_4^{(3)}\setminus e$. It appears that the
Ramsey numbers $r\big(H_G,K_n^{(3)}\big)$ have a very different
behavior depending on the bipartiteness of $G$.

\begin{proposition}

If $G$ is a bipartite graph, then there is a constant $c=c(G)$
such that $r(H_G,K_n^{(3)}) \leq n^{c}$. On the other hand, for
non-bipartite $G$, $r(H_G,K_n^{(3)}) \geq 2^{c'n}$ for an absolute
constant $c'>0$.
\end{proposition}
\begin{proof}
Let $G$ be a bipartite graph with $t$ vertices. The classical
result of K\"ovari, S\'os, and Tur\'an \cite{KST54} states that a
graph with $N$ vertices and at least $N^{2-1/t}$ edges contains
the complete bipartite graph $K_{t,t}$ with two parts of size $t$.
Therefore, any $3$-uniform hypergraph of order $N$ which contains
a vertex of degree at least $N^{2-1/t}$ contains also a copy of
$H_{K_{t,t}}$ and hence also $H_G$. Consider a red-blue
edge-coloring $C$ of the complete $3$-uniform hypergraph on
$N=(3n)^{2t}$ vertices, and let $m$ denote the number of red edges
in $C$. If $m \geq N^{3-1/t}$, then there is a vertex whose red
degree is at least $3m/N \geq N^{2-1/t}$, which by the above
remark gives a red copy of $H_G$. Otherwise, $m< N^{3-1/t}$ and we
can use a well known Tur\'an-type bound to find a large blue set
in coloring $C$. Indeed, it is well known (see, e.g., Chapter 3,
Exercise 3 in \cite{AS00}) that a $3$-uniform hypergraph with $N$
vertices and $m \geq N$ edges has an independent set (i.e., set
with no edges) of size at least $\frac{N^{3/2}}{3m^{1/2}}$. Thus,
the hypergraph of red edges has an independent set of size at
least $$\frac{N^{3/2}}{3m^{1/2}} >
\frac{N^{3/2}}{3(N^{3-1/t})^{1/2}}=\frac{1}{3}N^{1/(2t)} =n,$$
which clearly is a blue set.

To prove the second part of this proposition, we use a
construction of Erd\H{o}s and Hajnal mentioned in the
introduction. Suppose that $G$ is not bipartite, so it contains an
odd cycle with vertices $\{v_1,\ldots,v_{2i+1}\}$ and edges
$\{v_j,v_{j+1}\}$ for $1 \leq j \leq 2i+1$, where $v_{2i+2}:=v_1$.
We start with a tournament $T$ on $[N]$ with $N=2^{c'n}$ which
contains no transitive tournament of order $n$. As we already
mentioned, for sufficiently small $c'$, a random tournament has
this property with high probability. Color the triples from $[N]$
red if they form a cyclic triangle in $T$ and blue otherwise.
Clearly, this coloring does not contain a blue set of size $n$.
Suppose it contains a red copy of $H_G$. This implies that $T$
contains $2i+2$ vertices $v,u_1,\ldots,u_{2i+1}$ such that all the
triples $(v,u_j,u_{j+1})$ form a cyclic triangle. Then, the edges
$(v,u_j)$ and $(v,u_{j+1})$ have opposite orientation (one edge
oriented towards $v$ and the other oriented from $v$). Coloring
the vertices $u_j$ by $0$ or $1$ depending on the direction of
edge $(v,u_j)$ gives a proper $2$-coloring of an odd cycle,
contradiction.
\end{proof}

\subsection{Discrepancy in hypergraphs}
Despite the fact that Erd\H{o}s \cite{E90} (see also the book
\cite{CG98}) believed $r_3(n,n)$ is closer to $2^{2^{cn}}$,
together with Hajnal \cite{EH89} they discovered the following
interesting fact about hypergraphs which maybe indicates the
opposite. They proved that there are $c,\epsilon>0$ such that
every $2$-coloring of the triples of an $N$-set contains a set of
size $s> c(\log N)^{1/2}$ which contains at least
$(1/2+\epsilon){s \choose 3}$ $3$-sets in one color. That is, the
set of size $s$ deviates from having density $1/2$ in each color
by at least some fixed positive constant. Erd\H{o}s further
remarks that he would begin to doubt that $r_3(n,n)$ is
double-exponential in $n$ if one can prove that in any
$2$-coloring of the triples of the $N$-set, contains some set of
size $s=c(\eta)(\log N)^{\epsilon}$ for which at least $(1-\eta){s
\choose 3}$ triples have the same color. We prove the following
result, which demonstrates this if we allow $\epsilon$ to decrease
with $\eta$.

\begin{theorem} \label{highdensity} For $\eta>0$ and all positive integers $r$
and $k$, there is a constant $\beta=\beta(r,k,\eta)>0$ such that
every $r$-coloring of the $k$-tuples of an $N$-element set has a
subset of size $s> (\log N)^{\beta}$ which contains more than
$(1-\eta){s \choose k}$ $k$-sets in one color.
\end{theorem}

These results can be conveniently restated in terms of another
function introduced by Erd\H{o}s in \cite{E90}. Denote by
$F^{(k)}(N,\alpha)$ the largest integer for which it is possible
to split the $k$-tuples of a $N$-element set $S$ into two classes
so that for every $X \subset S$ with $|X| \geq F^{(k)}(N,\alpha)$,
each class contains more than $\alpha{|X| \choose k}$ $k$-tuples
of $X$. Note that $F^{(k)}(N,0)$ is essentially the inverse
function of the usual Ramsey function $r_k(n,n)$. It is easy to
show that for $0 \leq \alpha <1/2$, $$c(\alpha)\log N <
F^{(2)}(N,\alpha)<c'(\alpha)\log N.$$ As Erd\H{o}s points out, for
$k \geq 3$ the function $F^{(k)}(N,\alpha)$ is not well
understood. If $\alpha=1/2-\epsilon$ for sufficiently small
$\epsilon>0$, then the result of Erd\H{o}s and Hajnal from the
previous paragraph (for general $k$) demonstrates
$$c_k(\epsilon)\left(\log N
\right)^{1/(k-1)}<F^{(k)}(N,\alpha)<c'_k(\epsilon)\left(\log N
\right)^{1/(k-1).}$$ On the other hand, since $F^{(k)}(N,0)$ is
the inverse function of $r_k(n,n)$, then the old conjecture of
Erd\H{o}s, Hajnal, and Rado would imply that $$c_1\log_{(k-1)}
N<F^{(k)}(N,0)<c_2\log_{(k-1)} N,$$ where we recall that
$\log_{(t)} N$ denotes the $t$ times iterated logarithm function.
Assuming the conjecture, as $\alpha$ increases from $0$ to $1/2$,
$F^{(k)}(N,\alpha)$ increases from $\log_{(k-1)} n$ to $(\log
N)^{(1/(k-1)}$. Erd\H{o}s \cite{CG98} asked (and offered a \$500
cash reward) if the change in $F^{(k)}(N,\alpha)$ occurs
continuously, or there are jumps? He suspected the only jump
occurs at $\alpha=0$. If $\alpha$ is bounded away from $0$,
Theorem \ref{highdensity} demonstrates that $F^{(k)}(N,\alpha)$
already
 grows as some power of $\log N$. That is, for each $\alpha>0$ and
$k$ there are $c,\epsilon>0$ such that $F^{(k)}(N,\alpha)>c(\log
N)^{\epsilon}$.

We will deduce Theorem \ref{highdensity} from a result about the
$r$-color Ramsey number of a certain $k$-uniform hypergraph with
$n$ vertices and edge density almost one. The Ramsey number
$r(H;r)$ of a $k$-uniform hypergraph $H$ is the minimum $N$ such
that every $r$-edge-coloring of the $k$-tuples of a $N$-element
set contains a monochromatic copy of $H$. The blow-up
$K^{(k)}_{\ell}(n)$ is the $k$-uniform hypergraph whose vertex set
consists of $\ell$ parts of size $n$ and whose edges are all
$k$-tuples that have their vertices in some $k$ different parts.
Note that $K^{(k)}_{\ell}(n)$ has $\ell n$ vertices and ${\ell
\choose k}n^k \geq \left(1-{k \choose 2}/\ell\right){ln \choose
k}$ edges. In particular, as $\ell$ grows with $k$ fixed, the edge
density of $K^{(k)}_{\ell}(n)$ goes to $1$. Therefore, Theorem
\ref{highdensity} is a corollary of the following result.
\begin{theorem}\label{thmF}
For all positive integers $r,k,\ell$, there is a constant
$c=c(r,k,\ell)$ such that $$r\big(K^{(k)}_{\ell}(n);r\big) \leq
e^{c n^{\ell}}.$$
\end{theorem}
\begin{proof}
Consider an $r$-coloring of ${[N] \choose k}$ with $N
=e^{cn^{\ell-1}}$ and $c=\left(2r\cdot {t \choose
\ell}\right)^{\ell-1}$, where $t$ is the $r$-color Ramsey number
$r(K^{(k)}_{\ell};r)$. The proof uses a simple trick which appears
in \cite{E7172} and (see also \cite{KR06}). By definition, every
vertex subset of size $t$ contains a monochromatic set of size
$\ell$. Since each monochromatic set of size $\ell$ is contained
in ${N-\ell \choose t-\ell}$ subsets of size $t$, the number of
monochromatic sets of size $\ell$ is at least $${N \choose
t}/{N-\ell \choose t-\ell}={t \choose \ell}^{-1}{N \choose
\ell}.$$ By the pigeonhole principle, there is a color $1 \leq i
\leq r$ for which there are at least $\frac{1}{r}{t \choose
\ell}^{-1}{N \choose \ell}$ monochromatic sets of size $\ell$ in
color $i$. Define the $\ell$-uniform hypergraph $G$ with vertex
set $[N]$ whose edges consist of the monochromatic sets of size
$\ell$ in color $i$ in our $r$-coloring. We have just shown that
hypergraph $G$ with $N$ vertices has at least $\frac{1}{r}{t
\choose \ell}^{-1}{N \choose \ell} \geq
\epsilon\frac{N^{\ell}}{\ell!}$ edges with
$\epsilon=\frac{1}{2r}{t \choose \ell}^{-1}$. A standard extremal
lemma for hypergraphs (see, e.g., \cite{E64}, \cite{N08})
demonstrates that any $\ell$-uniform hypergraph with $N$ vertices
and at least $\epsilon\frac{N^{\ell}}{\ell!}$ edges with $(\ln
N)^{-1/(\ell-1)} \leq \epsilon \leq \ell^{-3}$ contains a complete
$\ell$-uniform $\ell$-partite hypergraph with parts of size
$\lfloor \epsilon (\ln N)^{1/(\ell-1)}\rfloor$. (An $l$-uniform
hypergraph is {\it $l$-partite} if there is a partition of the
vertex set into $l$ parts such that each edge has exactly one
vertex in each part.) In particular, $G$ contains a complete
$\ell$-uniform $\ell$-partite hypergraph with parts of size
$\lfloor \epsilon (\ln N)^{1/(\ell-1)}\rfloor = n$, where we use
that $\epsilon=c^{-1/(\ell-1)}$. The vertices of this complete
$\ell$-uniform $\ell$-partite hypergraph with $n$ vertices in each
part in $G$ are the vertices of a monochromatic
$K^{(k)}_{\ell}(n)$ in color $i$, completing the proof.
\end{proof}

Finally we want to mention another problem of Erd\H{o}s related to
the growth of Ramsey numbers of complete $3$-uniform hypergraphs.
Erd\H{o}s \cite{E71} (see also \cite{E90} and \cite{CG98}) asked
the following problem.

\begin{question} Suppose $|S|=N$ and the triples from $S$ are split into two
classes. Does there exist a pair of subsets $A,B \subset S$ with
$|A|=|B| \geq c(\log N)^{1/2}$ such that all triples from $A \cup
B$ that hit both $A$ and $B$ are in the same class? \end{question}
Erd\H{o}s showed that the answer is yes under the weaker
assumption that only the triples with two vertices in $A$ and one
vertex in $B$ must be monochromatic. Although this question is
still open we would like to mention that the answer to it is no if
the triples of $S$ are split into four classes instead of two.
Indeed, in \cite{CFS08}, we found a $3$-uniform hypergraph $C_n$
on $n$ vertices which is much sparser than the complete hypergraph
$K_n^{(3)}$ and whose four-color Ramsey number satisfies
$r(C_n;4)>2^{2^{c_1n}}$. Let $V = \{v_1, \cdots, v_n\}$ be a set
of vertices and let $C_n$ be the $3$-uniform hypergraph on $V$
whose edge set is given by $\{v_i, v_{i+1}, v_j\}$ for all $1 \leq
i, j \leq n$. (Note that when $i = n$, we consider $i+1$ to be
equal to 1.) When $n$ is even, the vertices of $C_n$ can be
partitioned into two subsets $A$ and $B$ (with $v_i \in A$ if and
only if $i$ is even) of size $n/2$ such that all edges of $C_n$
hit both $A$ and $B$. Thus, a four-coloring of the triples of
$[N]$ with $N=2^{2^{c_1n}}$ and with no monochromatic copy of
$C_n$ also does not contain a pair $A,B \subset [N]$ with
$|A|=|B|=\frac{1}{2c_1}\log \log N$ such that all triples that hit
both $A$ and $B$ are in the same class.

\vspace{0.1cm} \noindent {\bf Acknowledgments.}\, The results in
Section 6.1 were obtained in collaboration with Noga Alon, and we
thank him for allowing us to include them here. We also thank N.
Alon and D. Mubayi for interesting discussions.


\begin{thebibliography}{}

\bibitem{AKS80}
{M. Ajtai, J. Koml\'os, and E. Szemer\'edi,} {A note on Ramsey
numbers,} {\it J. Combinatorial Theory, Ser. A} {\bf 29} (1980),
354--360.

\bibitem{AS00}
N. Alon and J. H. Spencer, {\bf The probabilistic method,} 2nd
ed., Wiley, 2000.


\bibitem{B08}
{T. Bohman,} {The Triangle-Free Process,} preprint.

\bibitem{CG98}
{F. Chung and R. Graham,} {\bf Erd\H{o}s on Graphs. His Legacy of
Unsolved Problems}, A K Peters, Ltd., Wellesley, MA, 1998.

\bibitem{C08}
{D. Conlon,} {A new upper bound for diagonal Ramsey numbers,} {\it
Annals of Mathematics, to appear}.

\bibitem{CFS08}
{D. Conlon, J. Fox, and B. Sudakov,} {Ramsey numbers of sparse
hypergraphs,} {\it Random Structures and Algorithms, to appear}.

\bibitem{E47}
{P. Erd\H{o}s,} {Some remarks on the theory of graphs,} {\it Bull.
Amer. Math. Soc.} {\bf 53} (1947), 292--294.

\bibitem{E64} {P. Erd\H{o}s,} {On extremal problems of graphs and generalized
graphs,} {\it Israel J. Math} {\bf 2} (1964), 183--190.

\bibitem{E71} {P. Erd\H{o}s,}  {Topics in combinatorial analysis,} Proceedings
of the Second Louisiana Conference on Combinatorics, Graph Theory
and Computing (Louisiana State Univ., Baton Rouge, 1971) , pp.
2--20, Louisiana State University, Baton Rouge, LA, 1971.

\bibitem{E7172}
{P. Erd\H{o}s,} {On some extremal problems on $r$-graphs,} {\it
Discrete Math} {\bf 1} (1971/72), 1--6.

\bibitem{E90} {P. Erd\H{o}s,} {Problems and results on graphs and hypergraphs:
similarities and differences}, in {\it Mathematics of Ramsey
theory,} Algorithms Combin., Vol. 5 (J. Ne\v{s}et\u{r}il and V.
R\"odl, eds.) 12--28. Berlin: Springer-Verlag, 1990.

\bibitem{EH72}
{P. Erd\H{o}s, A. Hajnal,} {On Ramsey like theorems,} Problems and
results, Combinatorics (Proc. Conf. Combinatorial Math., Math.
Inst., Oxford, 1972) , pp. 123--140, Inst. Math. Appl.,
Southend-on-Sea, 1972.

\bibitem{EH89}
{P. Erd\H{o}s, A. Hajnal,} Ramsey-type theorems, {\it Discrete Appl.
Math.} {\bf 25} (1989), 37--52.

\bibitem{EHR65}
{P. Erd\H{o}s, A. Hajnal, R. Rado,} {Partition relations for
cardinal numbers,} {\it Acta Math. Acad. Sci. Hungar.} {\bf 16}
(1965), 93--196.

\bibitem{ER52}
{P. Erd\H{o}s, R. Rado,} {Combinatorial theorems on classifications
of subsets of a given set,} {\it Proc. London Math. Soc.} {\bf 3}
(1952), 417--439.

\bibitem{ES35}
P. Erd\H{o}s and G. Szekeres,
\newblock A combinatorial problem in geometry,
\newblock {\it Compositio Math.} {\bf 2} (1935), 463--470.

\bibitem{FS02}
R. J. Faudree and R. H. Schelp, A survey of results on size Ramsey
numbers, Paul Erd\H{o}s and his mathematics, II (Budapest, 1999),
Bolyai Soc. Math. Stud., Vol. 11, J\'anos Bolyai Math. Soc.,
Budapest, 2002, pp. 291--309.

\bibitem{FKRRT03}
{E. Friedgut, Y. Kohayakawa, V. R\"odl, A. Ruci\'nski, and P.
Tetali,} {Ramsey games against a one-armed bandit,} {\it Combin.
Prob. Comp.} {\bf 12} (2003), 515--545.

\bibitem{GRS90}
{R.L. Graham, B.L. Rothschild, J.L. Spencer,} {\bf Ramsey theory},
2nd edition, {John Wiley \& Sons} (1980).

\bibitem{K95}
{J. H. Kim,} {The Ramsey number $R(3,t)$ has order of magnitude
$t^2/\log t$,} {\it Random Structures and Algorithms} {\bf 7}
(1995), 173--207.


\bibitem{KR06}
{A. V. Kostochka and V. R\"odl,} {On Ramsey numbers of uniform
hypergraphs with given maximum degree,} {\it J. Combin. Theory Ser.
A} {\bf 113} (2006), 1555--1564.

\bibitem{KST54}
T. K\"ovari, V. S\'os, and P. Tur\'an, On a problem of K.
Zarankiewicz, {\it Colloq Math.} {\bf 3} (1954), 50--57.

\bibitem{KR05}
{A. Kurek and A. Ruci\'nski,} {Two variants of the size Ramsey
number,} {\it Discuss. Math. Graph Theory} {\bf 25} (2005),
141--149.

\bibitem{L07}
{L. Lov\'asz,} {\bf Combinatorial problems and exercises}, Corrected
reprint of the 1993 second edition. AMS Chelsea Publishing,
Providence, RI, 2007.

\bibitem{N08}
{V. Nikiforov,} {Complete $r$-partite subgraphs of dense
$r$-graphs,} preprint.

\bibitem{R30}
{F.P. Ramsey,} {On a problem of formal logic,} {\it Proc. London
Math. Soc. Ser. 2} {\bf 30} (1930), 264--286.


\bibitem{S77}
{J. Spencer,} {Asymptotic lower bounds for Ramsey functions,} {\it
Discrete Math.} {\bf 20} (1977/78), 69--76.



\end{thebibliography}
\end{document}